\documentclass{amsart}[12pt]
\usepackage{amsmath, amsthm, amscd, amsfonts,}
\usepackage{graphicx,float}

\usepackage{amssymb}

\usepackage{pb-diagram}

\newcommand{\Card}{\ensuremath{\text{Card}}}

\DeclareMathOperator{\Col}{Col}

\DeclareMathOperator{\GCH}{GCH}

\DeclareMathOperator{\KH}{KH}

\DeclareMathOperator{\ran}{ran}

\def\MPB{{\mathbb{P}}}

\def\ds{\displaystyle}

\def\sse{\subseteq}

\def\l{\lambda}
\def\lan{\langle}
\def\ran{\rangle}

\def\lora{\longrightarrow}

\newtheorem{theorem}{Theorem}[section]
\newtheorem{lemma}[theorem]{Lemma}

\newtheorem{question}[theorem]{Question}
\newtheorem{definition}[theorem]{Definition}

\newtheorem{remark}[theorem]{Remark}
\newtheorem{claim}[theorem]{Claim}
\numberwithin{equation}{section}

\usepackage{pb-diagram}

\def\l{\lambda}
\def\ds{\displaystyle}
\def\sse{\subseteq}

\def\lora{\longrightarrow}
\def\lan{\langle}
\def\ran{\rangle}

\def\rmark{\mbox{$\rm\bf\rule{0.06em}{1.45ex}\kern-0.05em R$}}
\def\pmark{\mbox{$\rm\bf\rule{0.06em}{1.45ex}\kern-0.05em P$}}
\def\nmark{\mbox{$\rm\bf\rule{0.06em}{1.45ex}\kern-0.05em N$}}
\def\vdash{\mbox{$\rm\| \kern-0.13em -$}}

\begin{document}

\title[Independence of Higher Kurepa Hypotheses]{Independence of Higher Kurepa Hypotheses }

\author[Sy-David Friedman and M. Golshani]{Sy-David Friedman and Mohammad Golshani $^1$}

\thanks{$^1$This work was done when the second author was at the Kurt G\"odel
Research Center. He would like to thank Prof. Friedman for his
inspiration and encouragement. The authors wish to thank the Austrian
Research Fund (FWF) for its generous support through Project P
21968-N13.} \maketitle




\begin{abstract}
We study the Generalized Kurepa Hypothesis introduced by Chang. We
show that relative to the existence of an inaccessible cardinal
the Gap-$n$-Kurepa hypothesis does not follow from the
Gap-$m$-Kurepa hypothesis for $m$ different from $n$. The use of
an inaccessible is necessary for this result.

\end{abstract}
\maketitle

\section{Introduction}

In this paper we study the Generalized Kurepa Hypothesis
introduced by Chang (see Chapter VII of [1]). We show that
relative to the existence of an inaccessible cardinal the
Gap$-n-$Kurepa hypothesis does not follow from the Gap$-m-$Kurepa
hypothesis for $m$ different from $n$. The use of an inaccessible
is necessary for this result.

\begin{definition} $(a)$ For infinite cardinals $\l<{\kappa}$, a
$\KH({\kappa},\l)-$family is a family $\mathcal{F}$ of subsets of
${\kappa}$ such that:

$(i)$ $\Card({\mathcal{F}})\geq {\kappa}^+$,

$(ii)$ for all $x\in[{\kappa}]^\l$, $\Card({\mathcal{F}}\upharpoonright  x)\leq\l$,
where ${\mathcal{F}}\upharpoonright  x=\{t\cap x: t\in{\mathcal{F}}\}$.

We say $\KH({\kappa},\l)$ holds if such a family exists.

$(b)$ For infinite cardinals $\l \leq{\kappa}$, a
$\KH({\kappa},<\l)-$family is a family $\mathcal{F}$ of subsets of
${\kappa}$ such that:

$(i)$ $\Card({\mathcal{F}})\geq {\kappa}^+$,

$(ii)$ for all $x\in[{\kappa}]^{< \l}$, $\Card({\mathcal{F}}\upharpoonright x)\leq
\Card(x)+ \aleph_0$.

We say $\KH({\kappa},<\l)$ holds if such a family exists.

$(c)$ Let $n\geq 1$, $n$ finite. By the  Gap$-n-$Kurepa hypothesis we mean the
following statement: for all infinite cardinals ${\lambda}$,
$\KH({\lambda}^{+n},{\lambda})$ holds.
\end{definition}

The following is well-known (see [1], Chapter VII, Theorems 3.2 and 3.3).

\begin{theorem}(Jensen). If $V=L$, then $\KH({\kappa},< \l^+)$ (and hence
$\KH({\kappa},\l)$) holds for all infinite cardinals $\l<{\kappa},
\kappa$ regular.
\end{theorem}

In this paper we prove the following  theorem.

\begin{theorem} Let $n\geq 1$. The following are
equiconsistent:

$(a)$ There exists an inaccessible cardinal,

$(b)$ $\GCH+$ the Gap$-m-$Kurepa hypothesis holds for all $m\neq
n$, but the Gap$-n-$Kurepa hypothesis fails.
\end{theorem}

\begin{remark} Our proof shows that if $\l<{\kappa}$ are
infinite cardinals, $\kappa$ regular and $\KH(\kappa, \l)$ fails,
then $\kappa^+$ is inaccessible in $L$ (see Lemma 3.1).
\end{remark}

\begin{remark} $(b)$ of the above Theorem can be strengthened to
the Gap$-m-$Kurepa hypothesis holds for all $m\neq n$, but
$\KH({\aleph}_n,{\aleph}_0)$ fails (see Lemma 2.7).
\end{remark}

\section{Proof of Con$(a)$ implies Con$(b)$}

 In this section we show that
if there exists an inaccessible cardinal, then in a forcing
extension of $L$, the
 Gap$-m-$Kurepa hypothesis holds for all $m\neq n$, but the  Gap$-n-$Kurepa hypothesis fails, where $n\geq 1$ is a fixed natural number.

 From now on assume that $V=L$, and let ${\kappa}$ be an inaccessible cardinal. We consider two cases.

 {\bf Case 1.} $n=1$.

 Let $\mathbb{P}=\Col(\omega_1,<{\kappa})$ be the Levy collapse with countable conditions
 which converts ${\kappa}$ into $\omega_2$, and let $G$ be $\mathbb{P}$-generic over $L$.

\begin{lemma} The following hold in $L[G]:$

$(a)$ $\KH(\aleph_1,\aleph_0)$ fails,

$(b)$ The Gap$-m-$Kurepa hypothesis holds for all $m\geq 2$.
\end{lemma}

\begin{proof} $(a)$ is a well known result of Silver (see [7], or [2] Lemma 20.4).

$(b)$ Let $m\geq 2$, and let $\l$ be an infinite cardinal in
$L[G]$. Let $\mu=(\l^{+m})^{L[G]}$. By Theorem 1.2, there is a
$\KH(\mu, \l)$ family $\mathcal{F}$ in $L$. We show that it remains
a $\KH(\mu, \l)$ family in $L[G]$. Clearly $\Card(\mathcal{F})=
\mu^{+L}=(\l^{+m+1})^{L[G]}$. Suppose $x\in([{\mu}]^\l)^{L[G]}$.

Note that $\MPB$ is $\kappa-c.c.$ and $\omega_1-$closed,  and in $L[G], \kappa$ becomes $\omega_2.$ Thus
it is easily seen that infinite sets in $L[G]$ are covered by sets of the same cardinality which belong to the ground
model $L,$ in particular
there is a set $y \subseteq \mu$ in $L$ such that $x
\subseteq y$ and $x$ and $y$ have the same cardinality in $L[G]$.
If $\l \neq \aleph_1,$ then $y$ has $L-$cardinality $\l,$ hence in
$L, \Card(\mathcal{F}\upharpoonright y) \leq \l.$ It follows that in $L[G],
\Card(\mathcal{F}\upharpoonright x) \leq \Card(\mathcal{F}\upharpoonright y) \leq \l.$ If $\l =
\aleph_1,$ then $y$ has $L-$cardinality less than $\kappa,$ hence
in $L, \Card(\mathcal{F}\upharpoonright y) < \kappa.$ It follows that in $L[G],
\Card(\mathcal{F}\upharpoonright y) \leq \aleph_1,$ and hence in $L[G],
\Card(\mathcal{F}\upharpoonright x) \leq \Card(\mathcal{F}\upharpoonright y) \leq
\aleph_1=\l.$
\end{proof}

{\bf Case 2.} $n\geq 2$.

For each $i$, $0< i<n$, fix an injection $J_i:[\omega_n]^{\leq
\omega_i}\lora \omega_n$. Let
$\mathbb{R}=\mathbb{P}\times\ds\prod_{0<i<n}\mathbb{Q}_i$, where
the forcing notions $\mathbb{P}$ and $\mathbb{Q}_i$, $0<i<n$, are
defined as follows.

 $\mathbb{P}=\Col(\omega_n,<{\kappa})$ is the
Levy collapse with conditions of size $<\omega_n$ which converts
${\kappa}$ into $\omega_{n+1}$.

 $\mathbb{Q}_i$, $0<i<n$, is the set of triples
 $p=(X_p,\mathcal{F}_p,g_p)$ such that:

$(i-1)$ $X_p$ is a subset of $\omega_n$ of size $\leq\omega_i$,

$(i-2)$ $\mathcal{F}_p$ is a subset of $^{X_p} 2$ of size
$\leq\omega_i$,

$(i-3)$ $g_p$  is a $1-1$ function from a subset of ${\kappa}$ into
$\mathcal{F}_p$,

$(i-4)$ $\mathcal{F}_p$ is $\omega_i-$closed in the following sense:
If $t \in$ $^{X_p} 2$ and $\langle X_\xi:\xi<\omega_{i-1}\rangle$
is a sequence of subsets of $X_p$ such that for all
$\xi<\omega_{i-1}$, $J_i(X_\xi)\in X_p$ and $t\upharpoonright X_\xi\in
\mathcal{F}_p\upharpoonright X_\xi$, then there is $s\in \mathcal{F}_p$ such
that $s\upharpoonright X=t\upharpoonright X$ and $s\upharpoonright (X_p\setminus X)=0\upharpoonright (X_p\setminus X)$ (=the zero function
on $X_p\setminus X$), where $X=\ds\bigcup_{\xi<\omega_{i-1}}X_\xi$.

For $p,q\in\mathbb{Q}_i$, let $p\leq q$ ($p$ is an extension of
$q$) iff:

$(i-5)$ $X_p\supseteq X_q$,

$(i-6)$ $\mathcal{F}_q= \mathcal{F}_p\upharpoonright X_q$,

$(i-7)$ $dom(g_p)\supseteq dom(g_q),$

$(i-8)$ for all $\alpha\in dom(g_q)$, $g_q(\alpha)=g_p(\alpha)\upharpoonright
X_q$.

We show that in the generic extension by $\mathbb{R}$, the
Gap$-m-$Kurepa hypothesis holds for all $m\neq n$, but the
Gap$-n-$Kurepa hypothesis fails.

\begin{lemma} $(a)$ $\mathbb{P}$ is $\omega_n$-closed,

$(b)$ $\mathbb{P}$ satisfies the ${\kappa}$-c.c.,

$(c)$ Let $0<i<n$. Then $\mathbb{Q}_i$ is $\omega_{i+1}-$closed
modulo $J_i$ in the following sense: If $\langle p_\xi:
\xi<\lambda\rangle$, $\lambda\leq \omega_i$, is a descending
sequence of conditions in $\mathbb{Q}  _i$ such that for all
$\xi<\lambda$, $J_i(X_{p_\xi})\in X_{p_{\xi+1}}$, then there is a
condition $p\in\mathbb{Q}_i$ which extends all of the $p_\xi$'s,
$\xi<\lambda$. Furthermore if $\l<\omega_i$, then $p$ can be
chosen to be the greatest lower bound of the $p_\xi$'s,
$\xi<\lambda$.

$(d)$ Let $0<i<n$. Then $\mathbb{Q}_i$ has the $\omega_{i+2}-$c.c.
\end{lemma}
\begin{proof} $(a)$ and $(b)$ are well known results of Levy (see
[2], Lemma 20.4 ). We prove $(c)$ and $(d)$.

$(c)$ Fix  $0<i<n$, and let $\langle p_\xi: \xi<\lambda\rangle$ be
as above. To simplify the notation let
$p_\xi=(X_\xi,\mathcal{F}_\xi,g_\xi)$, $\xi<\lambda$. We consider
two cases.

{\bf Case 1.} $\lambda<\omega_i$.

Let $p=(X,\mathcal{F},g)$, where:

\begin{itemize}
 \item $X=\ds\bigcup_{\xi<\lambda}X_\xi$,
\item $\mathcal{F}$ is the least subset of $^{X} 2$ such that if
$t\in$$^{X}2$ and for all $\xi<\lambda$, $t\upharpoonright X_\xi\in
\mathcal{F}_\xi$ then $t\in \mathcal{F}$, and $\mathcal{F}$ is
$\omega_i-$closed in the sense of $(i-4)$, \item
$dom(g)=\ds\bigcup_{\xi<\lambda} dom(g_\xi)$, \item
 for all
$\alpha\in dom(g)$, $g(\alpha)=\bigcup\{g_\xi(\alpha):\xi<\lambda,
\alpha\in dom(g_\xi)\}$.
\end{itemize}

It is easy to show that $p\in\mathbb{Q}_i$ and that $p$ is the
greatest lower bound for the sequence $\langle p_\xi:
\xi<\lambda\rangle$ .

{\bf Case 2.} $\lambda=\omega_i$.

Let $p=(X,\mathcal{F},g)$, where:
\begin{itemize}
\item $X=\ds\bigcup_{\xi<\lambda}X_\xi$, \item
$dom(g)=\ds\bigcup_{\xi<\lambda} dom(g_\xi)$, \item for all
$\alpha\in dom(g)$, $g(\alpha)=\bigcup\{g_\xi(\alpha):\xi<\lambda,
\alpha\in dom(g_\xi)\}$, \item $\mathcal{F}$ is the least subset of $^{X}2$ such that $ran(g) \cup \{t\upharpoonright X_{\xi} \cup 0\upharpoonright(X\setminus X_{\xi}): t \in X_{\xi} \} \subseteq \mathcal{F}$ and  $\mathcal{F}$ is $\omega_i-$closed in the sense of $(i-4).$
\end{itemize}
Then it is easy to show that  $p\in\mathbb{Q}_i$ and that $p$ is a
lower bound for the sequence $\langle p_\xi: \xi<\lambda\rangle$.

$(d)$ Fix $0<i<n$. Suppose that $\mathbb{Q}_i$ does not satisfy
the $\omega_{i+2}-$c.c. Let $A$ be a maximal antichain in
$\mathbb{Q}_i$ of size $\geq \omega_{i+2}$. By a $\Delta$-system
argument we can assume that
\begin{itemize}
\item The sequence $\langle X_p: p\in A\rangle$ forms a
$\Delta$-system with root $X$. \item The sequence $\langle
dom(g_p): p\in A\rangle$ forms a $\Delta-$system with root $D$.
\item For all $p \neq q$ in $A, g_p\upharpoonright D= g_q\upharpoonright D$ and
$\mathcal{F}_p\upharpoonright X= \mathcal{F}_q\upharpoonright X.$
\end{itemize}
Let $\theta$ be large regular, and let $M$ be an elementary
submodel of $H(\theta)$ of size $\omega_{i+1}$ which is closed
under
 $\omega_i-$sequences and such that $\mathbb{Q}_i, X,D,A\in M$.
 Pick $q\in A\setminus M$ and let $q\upharpoonright M=(X_q\upharpoonright M, \mathcal{F}_q\upharpoonright M,g_q\upharpoonright M)$, where:

$\hspace{1cm}\bullet$ $X_q\upharpoonright M=X_q\cap M$,

$\hspace{1cm}\bullet$ $\mathcal{F}_q\upharpoonright M=\{t\upharpoonright (X_q\cap M):t\in
F_q\}$,

$\hspace{1cm}\bullet$ $dom(g_q\upharpoonright M)=dom(g_q)\cap M$,

$\hspace{1cm}\bullet$ for all $\alpha\in dom(g_q\upharpoonright M),(g_q\upharpoonright
M)(\alpha)=g_q(\alpha)\upharpoonright (X_q\cap M)$.

Then $q\upharpoonright M \in \mathbb{Q}_i\cap M$. Extend this condition to a
condition $p\in\mathbb{Q}_i\cap M$ which extends an element $r\in
A$. We show that $p$ and $q$ and hence $r$ and $q$ are compatible,
which is impossible since $r,q\in A$.

Fix $s_0 \in \mathcal{F}_p, t_0 \in \mathcal{F}_q$. Define $X,
\mathcal{F}$ and $g$ as follows:
\begin{itemize}
\item $X= X_p \cup X_q,$ \item $\mathcal{F}$ is the least subset
of $^{X}2$ such that $\{s\upharpoonright X_p \cup t\upharpoonright (X_q \setminus M): s \in
\mathcal{F}_p, t \in \mathcal{F}_q \} \subseteq \mathcal{F},$ and
$\mathcal{F}$ is $\omega_i-$closed in the sense of $(i-4)$, \item
$dom(g)=dom(g_p) \cup dom(g_q),$ \item  $g(\alpha) = \left\{
\begin{array}{l}
       g_p(\alpha)\upharpoonright X_P \cup g_q(\alpha)\upharpoonright (X_q \setminus M) \hspace{1.65cm} \text{ if } \alpha \in domg_q \cap M,\\
       g_p(\alpha)\upharpoonright X_P \cup t_0\upharpoonright (X_q \setminus M) \hspace{2.2cm} \text{ if }  \alpha \in
       domg_p \setminus domg_q,\\
       g_q\upharpoonright X_q \cup s_0\upharpoonright (X_p \setminus X_q) \hspace{2.7cm} \text{ if } \alpha
       \in domg_q \setminus M.

     \end{array} \right.$
\end{itemize}

Then $(X, \mathcal{F}, g) \in \mathbb{Q}_i$ and it extends both of
$p$ and $q$.
\end{proof}

Let $K=G\times\ds\prod_{0<i<n} H_i$ be
$\mathbb{R}=\mathbb{P}\times \ds\prod_{0<i<n}\mathbb{Q}_i$ generic
over $L$. It follows from the above lemma that

$\hspace{1cm}\bullet$ $\omega_i^{L[K]}=\omega_i^L$ for all $i\leq
n$.

$\hspace{1cm}\bullet$ $\omega_{n+1}^{L[K]}={\kappa}^L$.

\begin{lemma} In $L[K]$, the Gap$-m-$Kurepa hypothesis holds
for all $m\neq n$.
\end{lemma}

\begin{proof} First show that $\KH(\aleph_n,\aleph_i)$ holds in
$L[K]$, for all $0<i<n$.

\begin{claim} Let $0<i<n$. Forcing with $\mathbb{Q}_i$ adds a
family $\mathcal{F}\sse~ ^{\omega_n}2$ such that

$(a)$ $\Card(\mathcal{F})={\kappa},$

$(b)$ for all $X\in ([\omega_n]^{\omega_i})^L$, $\Card(\mathcal{F}\upharpoonright
X)\leq\aleph_i.$
\end{claim}
\begin{proof} By Lemma 2.2, $\mathbb{Q}_i$ is a cardinal preserving
forcing notion. It is easy to prove the following (where $H_i$ is
assumed to be a $\mathbb{Q}_i$-generic filter over $L$):
\begin{itemize}
\item $\bigcup\{X_p:p\in H_i\}=\omega_n$, \item
$\bigcup\{dom(g_p):p\in H_i\}={\kappa}$, \item for all $X\in
([\omega_n]^{\omega_i})^L$, there is some $p\in H_i$ with
$X_q\supseteq X$, \item if $\alpha<{\kappa}$, then
$g(\alpha):\omega_n\lora 2$, where
\[g(\alpha)=\bigcup\{g_p(\alpha):p\in H_i,\alpha\in dom(g_p)\}\]
\item if $\alpha<\beta<{\kappa}$, then $g(\alpha)\neq g(\beta)$.
\end{itemize}
Then $\mathcal{F}=\{g(\alpha):\alpha<{\kappa}\}$ is as required.
\end{proof}

\begin{claim} Infinite sets in $L[K]$ are covered by sets of the same cardinality which belong to the ground
model $L.$
\end{claim}
\begin{proof}
It is easily seen that any infinite set of ordinals from $L[K]$ is covered by a set of
ordinals of $L[G]$ of the same cardinality and that $L[K]$ and $L[G]$ have the same cardinals. On the other hand since $\MPB$ is $\kappa-c.c.$ and $\omega_n-$closed and in $L[G], \kappa$ becomes $\omega_{n+1},$ any infinite set of ordinals from $L[G]$ is covered by a set of
ordinals of $L$ of the same $L[G]-$cardinality. The result follows immediately.
\end{proof}

Now using the above Claim and the fact that
$\omega_i^{L[K]}=\omega_i^L$,  we can show that $\mathcal{F}$ is
in fact a $\KH(\aleph_n, \aleph_i)-$family in $L[K]$.

Next let $\l$ be an infinite cardinal, $m \neq n,$ and suppose
$\mu= (\l^{+m})^{L[K]}, \mu \neq \aleph_n.$ We show that $\KH(\mu,
\l)$ holds in $L[K].$

\begin{claim} $\KH(\mu, \l)$ holds in $L[G].$
\end{claim}
\begin{proof} If $\mu < \aleph_n,$ the claim follows from the facts
that $\KH(\mu, \l)$ holds in $L, (\mu^+)^L= (\mu^+)^{L[G]}$ and $L$
and $L[G]$ have the same $\mu-$sequences. If $\mu > \aleph_n,$ the
claim follows exactly as in the proof of Lemma
2.1$(b)$.
\end{proof}

Using the facts that $L[G]$ and $L[K]$ have the same cardinals and
any infinite set of ordinals from $L[K]$ is covered by a set of
ordinals of $L[G]$ of the same cardinality, we can immediately
conclude that $\KH(\mu, \l)$ holds in $L[K].$ The Lemma
follows.
\end{proof}

\begin{lemma} $\KH(\aleph_n,\aleph_0)$ fails in $L[K]$.
\end{lemma}

Before going into the details of the proof of Lemma 2.7, we
introduce some notions. Let $\l$ be  a regular cardinal,
$\aleph_n<\l<{\kappa}$. Define the following forcing notions

$\hspace{.5cm}$ $\mathbb{P}_\l=Col(\omega_n,<\l)$,

$\hspace{.5cm}$ $\mathbb{Q}_{i,\l}=$the set of all $p\in
\mathbb{Q}_i$ such that $dom(g_p)\sse\l$,

 $\hspace{.5cm}$ $\mathbb{R}_\l=\mathbb{P}_\l\times \ds\prod_{0<
i<n}\mathbb{Q}_{i,\l}$

Also let $K_\l=G_\l\times \ds\prod_{0<i<n}H_{i,\l}$ be
$\mathbb{R}_\l$-generic over $L$. Define $\pi_\l:\mathbb{R}\lora
\mathbb{R}_\l$ by
\[\pi_\l(\lan p,\lan(X_i,\mathcal{F}_i,g_i):0<i<n\ran\ran)=\lan p\upharpoonright \l, \lan(X_i,\mathcal{F}_i,g_i \upharpoonright  \l):0<i<n\ran\ran\]

\begin{claim} $\pi_\l$ is a projection, i.e.

$(a)$ $\pi_\l(1_{\mathbb{R}})=1_{\mathbb{R}_\l}$,

$(b)$ $\pi_\l$ is order preserving,

$(c)$ if $r_0\in\mathbb{R}_\l$, $r_1\in\mathbb{R}$ and $r_0\leq
\pi_\l(r_1)$, then there is some $r\leq r_1$ in $\mathbb{R}$ such
that $\pi_\l(r)\leq r_0$.
\end{claim}
\begin{proof} $(a)$ and $(b)$ are trivial. We prove $(c).$ Let
$r_j=\lan p_j,\lan(X_{i,j},\mathcal{F}_{i,j},g_{i,j}):0<
i<n\ran\ran,$ for $j=0,1.$ Then $r=\lan
p,\lan(X_i,\mathcal{F}_i,g_i):0< i<n\ran\ran$ is as required,
where:
\begin{itemize}
\item $p=p_0 \cup p_1 \upharpoonright (\kappa\setminus\l),$ \item $X_i=X_{i,0},$ \item
$\mathcal{F}_i$ is the least subset of $^{X_i}2$ such that
$\mathcal{F}_{i,o} \cup \{t \upharpoonright X_{i,1} \cup 0 \upharpoonright (X_{i,0}\setminus X_{i,1})\}
\subseteq \mathcal{F}_i,$ and $\mathcal{F}_i$ is $\omega_i-$closed
in the sense of $(i-4),$ \item $domg_i=domg_{i,0} \cup
domg_{i,1},$ \item  $g_i(\alpha) = \left\{ \begin{array}{l}
       g_{i,o}(\alpha)  \hspace{5cm} \text{ if } \alpha \in domg_{i,0},\\
       g_{i,1}(\alpha) \upharpoonright X_{i,1} \cup 0 \upharpoonright (X_{i,0}\setminus X_{i,1}) \hspace{1.2cm} \text{ if }  \alpha \in
       domg_{i,1}\setminus\l.

     \end{array} \right.$
 \end{itemize}
 \end{proof}

Let
\[(\mathbb{R}:\mathbb{R}_\l)=\{\lan p,\lan(X_i,\mathcal{F}_i,g_i):0< i<n\ran\ran\in\mathbb{R}:\pi_\l(\lan p,\lan(X_i,\mathcal{F}_i,g_i):0<i<n\ran\ran)\in K_\l\}.\]

It follows from Lemma 2.2 $(c)$ that

\begin{claim} $(\mathbb{R}:\mathbb{R}_\l)$ is countably
closed modulo the $J_i$'s, $0< i<n$, in the following sense: if
$\lan\lan p_m,\lan(X_{i,m},\mathcal{F}_{i,m},g_{i,m}):0<
i<n\ran\ran:m<\omega\ran$ is a descending sequence of conditions
in $(\mathbb{R}:\mathbb{R}_\l)$ such that for all $0<i<n$ and
$m<\omega$, $J_i(X_{i,m})\in X_{i,m+1}$, then this sequence has a
lower bound in $(\mathbb{R}:\mathbb{R}_\l)$.
\end{claim}
\begin{proof} For each $i, 0<i<n,$ the sequence
$\lan(X_{i,m},\mathcal{F}_{i,m},g_{i,m}):m < \omega \ran$ is a
descending sequence in $\mathbb{Q}_i$ modulo $J_i,$ thus by Lemma
2.2$(c)$ it has a greatest lower bound $(X_i, \mathcal{F}_i,
g_i).$ Let $r=\lan \ds\bigcup_{m<
\omega}p_m,\lan(X_i,\mathcal{F}_i,g_i):0< i<n\ran \ran$. Then $r$
is the greatest lower bound for the above sequence, and
$\pi_{\l}(r)$ is a lower bound for the sequence $\lan
\pi_{\l}(\lan
p_m,\lan(X_{i,m},\mathcal{F}_{i,m},g_{i,m}):0<i<n\ran
\ran):m<\omega\ran$. Note that the projection $\pi_{\l}$ just restricts the domain of functions involved in the condition to $\lambda$ and thus we can easily show that:

 \begin{itemize}
\item $\pi_{\l}(r)$ is in fact the greatest lower bound of the above sequence.

\item If $r'$ is compatible with all of $\lan p_m,\lan(X_{i,m},\mathcal{F}_{i,m},g_{i,m}):0<
i<n\ran\ran, m<\omega,$ then $r'$ is compatible with $\pi_{\l}(r).$
\end{itemize}
It then follows from the maximality of $K_{\l}$ that $\pi_{\l}(r) \in K_{\l},$ and
hence $r \in (\mathbb{R}:\mathbb{R}_\l).$ Thus $r$ is as required
\end{proof}

We are now ready to prove Lemma 2.7.
 Assume on the contrary that
$\KH(\aleph_n,\aleph_0)$ holds in $L[K]$. Suppose for simplicity
that $1_{\mathbb{R}} \vdash\ulcorner\dot{\mathcal{F}}$ is a
$\KH(\aleph_n,\aleph_0)$-family $\urcorner$.

Let $\mathcal{F}=\dot{\mathcal{F}}[K]$, and let $A=\lan
\mathcal{F}\upharpoonright X:X\in[\omega_n]^\omega\ran$. Choose $\l<{\kappa}$
regular such that $A\in L[K_\l]$. Let $b\in\mathcal{F}$ be such
that $b\not\in L[K_\l]$.

From now on we work in $L[K_\l]$ and force with
$(\mathbb{R}:\mathbb{R}_\l)$. Let $\dot{b}$ be an
$(\mathbb{R}:\mathbb{R}_\l)$-name for $b$, and let $r_0\in
(\mathbb{R}:\mathbb{R}_\l)$, $r_0=\lan
p_0,\lan(X_{i,0},F_{i,0},g_{i,0}):0<i<n\ran\ran$, be such that
\[r_0\vdash\ulcorner \dot{b}\in\dot{\mathcal{F}} \ \ and \ \ \dot{b}\not\in V\urcorner\]

It is easy to prove the following.

\begin{claim} For each $r\leq r_0$, $r=\lan
p,\lan(X_{i},F_{i},g_{i}):0<i<n\ran\ran$, there are two conditions
$r_1=\lan p_1,\lan(X_{i,1},F_{i,1},g_{i,1}):0<i<n\ran\ran$,
$r_2=\lan p_2,\lan(X_{i,2},F_{i,2},g_{i,2}):0< i<n\ran\ran$ and
some $\xi<\omega_n$ such that:

$(a)$ $r_1,r_2\leq r$,

$(b)$ $J_i(X_i)\in X_{i,m}$ for all $0<i<n$ and $m=1,2$,

$(c)$ $r_1\vdash\ulcorner \check{\xi}\in\dot{b}\urcorner$ iff
$r_2\vdash\ulcorner
\check{\xi}\not\in\dot{b}\urcorner$.\hfill$\Box$
\end{claim}

Using the above, we can construct a sequence $\lan r_s=\lan
p_s,\lan(X_{i,s},F_{i,s},g_{i,s}):0<i<n\ran\ran:s\in ~ ^{<\omega}
2\ran$ of conditions in $(\mathbb{R}:\mathbb{R}_\l)$ and a
sequence $\lan\xi_m:m<\omega\ran$ of elements of $\omega_n$ such
that the following hold:
\begin{itemize}
\item $r_{s*m}\leq r_s$, for each $s\in ~ ^{<\omega} 2$ and $m<2$,
\item  $J_i(X_{i,s})\in X_{i,s*m}$ for each $s\in ~ ^{<\omega} 2$,
$m<2$ and $0< i<n$, \item $r_{s*0}\vdash\ulcorner \check{\xi}_m\in
\dot{b}\urcorner$ iff $r_{s*1}\vdash\ulcorner \check{\xi}_m\not\in
\dot{b}\urcorner$, where $m$ is the length of $s$.
\end{itemize}
Let $X=\{\xi_m:m<\omega\}$, and for each $f\in$ $^{\omega}2$,
using Claim 2.9, let $r_f\in(\mathbb{R}:\mathbb{R}_\l)$ be an
extension of all of the
 $r_{f\upharpoonright m}$'s, $m<\omega$. For each $f$ as above, we can find some $q_f\leq r_f$ and some $b_f\in L[K_\l]$ such that
\[q_f\vdash\ulcorner \dot{b}\cap \check{X}=\check{b}_f\urcorner\]

Note that  $\mathcal{F}\upharpoonright X \supseteq \{b_f: f \in$ $^{\omega}2 \}$
and for $f\neq g$ in $^{\omega} 2$, we have $b_f\neq b_g$, and
hence $\mathcal{F}\upharpoonright X$ must have size at least $2^{\aleph_0}$
which is in contradiction with our assumption.

It follows that $\KH(\aleph_n,\aleph_0)$ fails in $L[K]$. This
completes the proof of Lemma 2.7.

\section{Proof of Con$(b)$ implies Con$(a)$}

Now we show that if $n\geq 1$,
and the Gap$-n-$Kurepa hypothesis fails, then there exists an
inaccessible cardinal in $L$. In fact we will prove the following
more general result.

\begin{lemma} Suppose that $\l<{\kappa}$ are infinite cardinals
such that $\kappa$ is regular, ${\kappa}^\l={\kappa}$ and
$\KH({\kappa},\l)$ fails. Then ${\kappa}^+$ is an inaccessible
cardinal in $L$.
\end{lemma}

The rest of this section is devoted to the prove of the above lemma. Assume on the contrary that the lemma fails. Thus we can find $X\sse {\kappa}$ such
that:
\begin{itemize}
\item $V$ and $L[X]$ have the same cardinals up to ${\kappa}^+$,
\item $([{\kappa}]^\l)^V=([{\kappa}]^\l)^{L[X]}$.
\end{itemize} It follows that a $\KH({\kappa},\l)$-family in $L[X]$
is a real $\KH({\kappa},\l)$-family, and hence $\KH({\kappa},\l)$
fails in $L[X]$. The following lemma gives us the required
contradiction.

\begin{lemma} Suppose that $V=L[X]$, where $X\sse {\kappa}$.
Then  $\KH({\kappa},\l)$ holds.
\end{lemma}

\begin{proof} Our proof is very similar to the proof of Theorem 2
in [3]. We give it for completeness. For each $x\in[{\kappa}]^\l$
let
\begin{center}
$M_x=$ the smallest $M\prec L_{\kappa}[X]$ such that
$x\cup\{x\}\cup(\l+1)\sse M.$
\end{center}
Let $\mathcal{F}=\{t\sse {\kappa}:\forall x\in[{\kappa}]^\l, t\cap
x\in M_x\}$. We show that $\mathcal{F}$ is a
$\KH({\kappa},\l)-$family. It suffices to show that
$\Card(\mathcal{F})\geq {\kappa}^+$. Suppose not. Let $C=\lan
t_\nu:\nu<{\kappa}\ran$ be an enumeration of $\mathcal{F}$
definable in $L_{{\kappa}^+}[X]$. By recursion on $\nu<{\kappa}$,
define a chain $\lan N_\nu:\nu<{\kappa}\ran$ of elementary
submodels of $L_{{\kappa}^+}[X]$ as follows:

$N_0=$ the smallest $N\prec L_{{\kappa}^+}[X]$ such that $\l \in
N$ and $N\cap {\kappa}\in {\kappa}$,

$N_{\nu+1}=$ the smallest $N\prec L_{{\kappa}^+}[X]$ such that
$N\cap {\kappa}\in {\kappa}$ and $N_\nu\cup\{N_\nu\}\sse N$,

$N_\delta=\ds\bigcup_{\nu<\delta}N_\nu$, if $\delta$ is a limit
ordinal.

For each $\nu<{\kappa}$ set $\alpha_\nu=N_\nu\cap {\kappa}$. Using
the condensation lemma for $L[X]$, we obtain an ordinal
$\beta_\nu$ and an isomorphism $\sigma_\nu$ such that
\[\sigma_\nu:\lan N_\nu,\in,N_\nu\cap X\ran \simeq\lan L_{\beta_\nu}[X\cap\alpha_\nu],\in,X\cap\alpha_\nu\ran.\]
Then:
\begin{itemize}
\item $\alpha_\nu<\beta_\nu<\alpha_{\nu+1}$, \item
$\sigma_\nu({\kappa})=\alpha_{\nu}$, \item $\sigma_\nu(X)=X\cap
\alpha_{\nu}$, \item $\sigma_\nu\upharpoonright \alpha_{\nu}=id\upharpoonright \alpha_\nu$,
\item $L_{\beta_\nu} [X\cap \alpha_{\nu}]\models\ulcorner
\alpha_\nu$ is a regular cardinal, and $\alpha_\nu$ is the largest
cardinal $\urcorner$.
\end{itemize}
Let $t=\{\beta_\nu:\beta_\nu\not\in t_\nu\}$. Clearly $t\neq
t_\nu$ for all $\nu<{\kappa}$, and hence $t\not\in\mathcal{F}$.
Let $x\in[{\kappa}]^\l$ be such that:
\begin{itemize}
\item $t\cap x\not\in M_x$, \item $\alpha=sup(x)$ is minimal.
\end{itemize}
It follows that $t\cap x$ is cofinal in $\alpha$, and hence
$\alpha=\alpha_\eta$ for some $\eta<{\kappa}$. We have
\[t\cap x=\{\beta_\nu\in x: \beta_\nu<\alpha_\eta \ \ and \ \ \beta_\nu\not\in t_\nu\cap \alpha_\eta\}\]
and thus $t\cap x$ is definable from $x$, $\lan \beta_\nu:
\nu<\eta\ran$ and $\lan t_\nu\cap \alpha_\eta: \nu<\eta\ran$. It
is clear that:
\begin{itemize}
\item  $x\in M_x$, \item  $\lan \beta_\nu: \nu<\eta\ran$ is
definable in $L_{\beta_\eta}[X\cap\alpha_\eta]$. \item
$\sigma_\eta(C)=\lan t_\nu\cap\alpha_\eta:\nu<\eta\ran$, and hence
$\lan t_\nu\cap \alpha_\eta: \nu<\eta\ran$ is definable in
$L_{\beta_\eta} [X\cap \alpha_{\eta}]$.
\end{itemize}
Clearly $X\cap\alpha_\eta\in M_x$. We show that $\beta_\eta\in
M_x$. It will follow that $t\cap x\in M_x$ which is a
contradiction. The proof is in a sequence of claims. Let $M=M_x$.

\begin{claim} $\mathcal{P}(\alpha_\eta)\cap M\not\sse
L_{\beta_\eta}[X\cap \alpha_\eta]$.
\end{claim}
\begin{proof} Suppose not. Since $cf(\alpha_\eta)=cf(x) \leq \l
<\alpha_{\eta}$, there is $a\in M$ such that $a\sse\alpha_\eta$ is
cofinal in $\alpha_\eta$ and has order type less than
$\alpha_\eta$. Then $a\in L_{\beta_\eta}[X\cap\alpha_\eta]$, and
hence $\alpha_\eta$ is not a regular cardinal in
$L_{\beta_\eta}[X\cap\alpha_\eta]$. A contradiction.
\end{proof}

For $l<\nu<{\kappa}$ set:
\begin{itemize} \item
$\alpha^{(\nu)}=\lan\alpha_\iota: \iota\leq \nu\ran$, \item
$\beta^{(\nu)}=\lan\beta_\iota: \iota\leq \nu\ran$, \item
$\sigma_{\iota\nu}=\sigma_{\nu}\sigma_{\iota}^{-1}:\lan
L_{\beta_\iota}[X\cap\alpha_\iota],\in,X\cap\alpha_\iota\ran
\lora \lan
L_{\beta_\nu}[X\cap\alpha_\nu],\in,X\cap\alpha_\nu\ran$,
\item  $\sigma^{(\nu)}=\lan \sigma_{\iota\nu}:\iota<\tau\leq
\nu\ran$.
\end{itemize}

\begin{claim} $\nu\in M\cap \eta$ implies $\alpha^{(\nu)},
\beta^{(\nu)}, \sigma^{(\nu)}\in M$.
\end{claim}
\begin{proof} First note that $\alpha_{\nu} \in M$ implies
$\alpha^{(\nu)} \in M,$ since $\lan \alpha_{\iota}: \iota< \nu
\ran$ is definable from $L_{\beta_\nu}[X\cap\alpha_\nu]$ the way
$\lan\alpha_\iota: \iota< \kappa \ran$ was defined from
$L_{\kappa^+}[X].$ It follows that $\nu\in M\cap \eta$ implies
$\alpha^{(\nu)} \in M$, since there is $\tau, \nu \leq \tau <
\eta$ such that $\alpha_{\tau} \in M$ and $\alpha_{\nu}=
\alpha^{\tau}(\nu) \in M.$ By similar arguments $\nu\in M\cap
\eta$ implies $ \beta^{(\nu)}, \sigma^{(\nu)}\in M$.
\end{proof}

We note that
\[ \lan\lan L_{\beta_\iota}[X\cap\alpha_\iota],\in,X\cap\alpha_\iota\ran_{\iota<\eta},\lan \sigma_{\iota\nu}\ran_{\iota<\nu<\eta}\ran\]
is a directed system of elementary embeddings, and if
 \[ \lan \lan U,E,Y\ran, \lan g_\iota\ran_{\iota<\eta} \ran\]
 is its direct limit, then:
\begin{itemize}
\item $\lan U,E,Y\ran\simeq\lan
L_{\beta_\eta}[X\cap\alpha_\eta],\in,X\cap\alpha_\eta\ran$,
\item $g_\iota:\lan
L_{\beta_\iota}[X\cap\alpha_\iota],\in,X\cap\alpha_\iota\ran\lora\lan
U,E,Y\ran $, \item If $f:\lan U,E,Y\ran\simeq\lan
L_{\beta_\eta}[X\cap\alpha_\eta],\in,X\cap\alpha_\eta\ran$,
then $\sigma_{\iota\eta}=fg_\iota$.
\end{itemize}
Now let $\pi:\lan M,\in,M\cap X\ran\simeq\lan
L_\delta[\tilde{X}],\in,\tilde{X}\ran$, where
$\tilde{X}=\pi[M\cap X]$. Let

\begin{itemize}
\item $\tilde{\alpha}^{(\nu)}=\pi(\alpha^{(\nu)})$, \item
$\tilde{\beta}^{(\nu)}=\pi(\beta^{(\nu)})$, \item
$\tilde{\sigma}^{(\nu)}=\pi(\sigma^{(\nu)})$, \item
$\tilde{\alpha}=\ds\bigcup_{\nu\in M\cap \eta}
\tilde{\alpha}^{(\nu)}$, \item $\tilde{\beta}=\ds\bigcup_{\nu\in
M\cap \eta} \tilde{\beta}^{(\nu)}$,
\item $\tilde{\sigma}=\ds\bigcup_{\nu\in M\cap \eta} \tilde{\sigma}^{(\nu)}$,\\
\end{itemize}
and
\begin{itemize}
\item $\tilde{\alpha}_{\iota}=\pi(\alpha_{\pi^{-1}(\iota)})$,
\item $\tilde{\beta}_\iota=\pi(\beta_{\pi^{-1}(\iota)})$,
\item $\tilde{\sigma}_{\iota\nu}=\pi(\sigma_{\pi^{-1}(\iota),\pi^{-1}(\nu)})$.\\
\end{itemize}
Now
\[ \lan\lan L_{\tilde{\beta}_\iota}[\tilde{X}\cap\tilde{\alpha}_\iota],\in,\tilde{X}\cap\tilde{\alpha}_\iota\ran_{\iota<\pi(\eta)},
\lan \tilde{\sigma}_{\iota\nu}\ran_{\iota<\nu<\pi(\eta)}\ran\] is
a directed system of elementary embeddings. Let
 \[ \lan \lan \tilde{U},\tilde{E},\tilde{Y}\ran, \lan \tilde{g}_\iota\ran_{\iota<\pi(\eta)} \ran\]
be its direct limit. Then
\begin{itemize}
\item $\tilde{g}_\iota:\lan
L_{\tilde{\beta}_\iota}[\tilde{X}\cap\tilde{\alpha}_\iota],\in,\tilde{X}\cap\tilde{\alpha}_\iota\ran\lora\lan
\tilde{U},\tilde{E},\tilde{Y}\ran $, \item There is an elementary
embedding $h$ such that the following diagram is commutative
\[\begin{array}{ccc}
 \lan L_{\beta_{\pi^{-1}(\iota)}}[X\cap\alpha_{\pi^{-1}(\iota)}],\in,{X}\cap{\alpha}_{\pi^{-1}(\iota)}\ran & \stackrel{g_{\pi^{-1}(\iota)}}{\lora}
 & \lan U,E,Y\ran \\ \pi^{-1} \uparrow && \uparrow h \\
 \lan L_{\tilde{\beta}_\iota}[\tilde{X}\cap\tilde{\alpha}_\iota],\in,\tilde{X}\cap\tilde{\alpha}_\iota\ran & \stackrel{\tilde{g}_{\iota}}{\lora} & \lan \tilde{U},\tilde{E},\tilde{Y}\ran\end{array}\]
\end{itemize}
It follows that $\lan \tilde{U},\tilde{E}\ran$ is
well founded. Let
\[\tilde{f}:\lan \tilde{U},\tilde{E},\tilde{Y}\ran\simeq\lan L_{\bar{\beta}}[\bar{X}],\in,\bar{X}\ran.\]
Also let
\begin{itemize}
\item $\bar{\sigma}_\iota=\tilde{f}\tilde{g}_\iota:\lan
L_{\tilde{\beta}_\iota}[\tilde{X}\cap\tilde{\alpha}_\iota],\in,\tilde{X}\cap\tilde{\alpha}_\iota\ran\lora
\lan L_{\bar{\beta}}[\bar{X}],\in,\bar{X}\ran$, \item
$\pi^* = fh \tilde{f}^{-1}:\lan
L_{\bar{\beta}}[\bar{X}],\in,\bar{X}\ran\lora
\lan L_{\beta_\eta}[{X}\cap{\alpha}_\eta],\in,{X}\cap{\alpha}_\eta\ran$.\\
\end{itemize}
Then
$\tilde{\sigma}_{\iota\tau}=\bar{\sigma}_\tau^{-1}\bar{\sigma}_\iota$
for $\iota<\tau<\pi(\eta)$, and the following diagram is
commutative
\[\begin{array}{ccc}
 \lan L_{\beta_{\pi^{-1}(\iota)}}[X\cap\alpha_{\pi^{-1}(\iota)}],\in,{X}\cap{\alpha}_{\pi^{-1}(\iota)}\ran & \stackrel{\sigma_{\pi^{-1}(\iota),\eta}}{\lora} & \lan L_{{\beta}_\eta}[{X}\cap{\alpha}_\eta],\in,{X}\cap{\alpha}_\eta\ran
 \\ \pi^{-1} \uparrow && \uparrow \pi^* \\
 \lan L_{\tilde{\beta}_\iota}[\tilde{X}\cap\tilde{\alpha}_\iota],\in,\tilde{X}\cap\tilde{\alpha}_\iota\ran\ & \stackrel{\bar{\sigma}_{\iota}}{\lora} & \lan L_{\bar{\beta}}[\bar{X}],\in,\bar{X} \ran\end{array}\]

Let $\bar{\alpha}$ be such that
$L_{\bar{\beta}}[\bar{X}]\models\ulcorner\bar{\alpha}$ is the
largest cardinal $\urcorner$.

\begin{claim} $(a)$ $\pi(\alpha_\eta)=\bar{\alpha}$,

$(b)$ $\pi^*(\bar{\alpha})=\alpha_\eta$,

$(c)$ $\pi^* \upharpoonright  \bar{\alpha}=id \upharpoonright \bar{\alpha}$.
\end{claim}
\begin{proof} $(a)$ follows easily from the facts that
$\bar{\alpha}=sup_{\iota < \eta}\tilde{\alpha}_{\iota},$
$\alpha_{\eta}=sup_{\iota \in M \cap \eta}\alpha_{\iota}$ and
$\pi^{-1}(\tilde{\alpha}_{\iota})=\alpha_{\pi^{-1}(\iota)}.$ $(b)$
follows from the choice of $\bar{\alpha}$ and the elementarily of
$\pi^*$. $(c)$ is trivial, as $\bar{\alpha} \subseteq
L_{\bar{\beta}}[\bar{X}].$
\end{proof}

Next we have

\begin{claim} If $a\sse \bar{\alpha}$ and $a\in
L_{\bar{\beta}}[\bar{X}]\cap L_\delta[\tilde{X}]$, then
$\pi^*(a)=\pi^{-1}(a)$.
\end{claim}
\begin{proof} Since $a\sse\bar{\alpha}$, $\pi^*(a)$,
$\pi^{-1}(a)\sse\alpha_{\eta}$, and hence
$\pi^*(a)=\ds\bigcup_{\nu\in M\cap\eta} \pi^*(a)\cap \nu=
\ds\bigcup_{\nu< \pi(\eta)} \pi^*(a\cap \nu) \stackrel{claim 3.5}{=}
\ds\bigcup_{\nu<\pi(\eta)} \pi^{-1}(a\cap
\nu)=\pi^{-1}(a).$
\end{proof}

\begin{claim} $\delta>\bar{\beta}$.
\end{claim}
\begin{proof} Suppose not. Then $\delta\leq\bar{\beta}$ and
$\pi^*\pi$ maps $M$ into $L_{{\beta}_\eta}[{X}\cap{\alpha}_\eta]$,
and by claim 3.6, $\pi^*\pi(a)=a$ for $a\sse\alpha_\eta$, $a\in M$.
It follows that ${\mathcal{P}}(\alpha_\eta)\cap M\sse
L_{{\beta}_\eta}[{X}\cap{\alpha}_\eta]$, which is in contradiction
with claim 3.3.
\end{proof}

It follows that $\bar{\beta}\in L_\delta[\tilde{X}]$ and hence
$\tilde{\beta}=\lan \tilde{\beta}_{\iota}: \iota < \pi(\eta) \ran
\in L_\delta[\tilde{X}],$ since $\tilde{\beta}$ is definable from
$L_{\bar{\beta}}[\bar{X}]$ as $\lan \tilde{\beta}_{\iota}: \iota <
\kappa \ran$ was defined from $L_{\kappa^+}[X]$. Similarly
$\tilde{\sigma}=\lan \tilde{\sigma}_{\iota, \nu}: \iota<\nu<
\pi(\eta) \ran \in L_\delta[\tilde{X}]$. It is easily seen that

\begin{claim}$(a)$   $\pi^{-1}(\tilde{\alpha})=\lan \alpha_\iota
: \iota<\eta\ran$,

$(b)$  $\pi^{-1}(\tilde{\beta})=\lan \beta_\iota :
\iota<\eta\ran$,

$(c)$  $\pi^{-1}(\tilde{\sigma})=\lan \sigma_{\iota \nu}:
\iota<\nu<\eta\ran$.\hfill$\Box$
\end{claim}
Now note that:
\begin{itemize}
\item $L_{\bar{\beta}}[\bar{X}]$ is the direct limit of
$L_{\tilde{\beta}_\iota}[\tilde{X}\cap\tilde{\alpha}_\iota]$,
$\tilde{\sigma}_{\iota\nu}, \iota<\nu<\pi(\eta)$, \item
$\pi^{-1}[\bar{X}]=X\cap \alpha_\eta$,
\item $\pi^{-1}[\tilde{X}\cap\tilde{\alpha}_\iota]=X\cap \alpha_\iota$,\\
\end{itemize}
and hence by elementarily of $\pi^{-1}$,
$L_{\pi^{-1}(\bar{\beta})}[X\cap \alpha_\eta]$ is the direct limit
of $L_{{\beta}_\iota}[X\cap \alpha_\iota]$, ${\sigma}_{\iota\nu},
\iota<\nu<\eta$.

It follows that $\pi^{-1}(\bar{\beta})=\beta_\eta\in M$. We are
done.
\end{proof}

\section{Open problems}

We close the paper with some remarks and open problems.

By the results of Vaught, Chang, Jensen (see [1], Chapter VIII)
and Silver (see [7]), it is consistent, relative to the existence
of an inaccessible cardinal, to have the Gap$-n-$transfer
principle with the failure of the gap$-(n+1)-$transfer principle
for $n=1$. The answer is unknown for $n>1$.

\begin{question} Let $n>1.$ Is it consistent to have the
Gap$-n-$transfer principle with the failure of the
Gap$-(n+1)-$transfer principle?
\end{question}
Another related question is

\begin{question} Let $n>1.$ Is it consistent to have $(\kappa,
n)-$morasses for each uncountable regular $\kappa$, but no
$(\omega_1, n+1)-$morasses?
\end{question}
\begin{remark} Assuming the existence of large cardinals, it is
possible to build a model of set theory in which there exists a
$(\kappa, 1)-$morass for each uncountable regular $\kappa$, but
there are no $(\omega_1, 2)-$morasses.
\end{remark}
In the literature the canonical counter-example to the
Gap$-1-$transfer principle is the non-existence of Special
Aronszajn trees (see [5]). T. Raesch, in his dissertation (see
[6]), showed that this principle can fail in the presence of such
trees. On the other hand the canonical counter-example to the
Gap$-2-$transfer principle is the non-existence of Kurepa trees
(see [7]). Inspired by the work of Raesch, Jensen produced,
relative to the existence of a Mahlo cardinal, a model in which
the Gap$-2-$transfer principle fails, while the Gap$-1-$Kurepa
hypothesis holds (see [4]). However the following is open.

\begin{question} Is it consistent relative to an inaccessible
cardinal to have the Gap$-1-$Kurepa Hypothesis but a failure of
the Gap $-2-$transfer principle?
\end{question}
\begin{remark} It is possible to show that the existence of an
$(\omega_2, 1)-$morasses implies $\KH(\aleph_2, <\aleph_2).$ Thus
in our model, for $n=2,$ the Gap$-1-$Kurepa hypothesis holds,
while in it there are no $(\omega_2, 1)-$morasses.
\end{remark}
\begin{question} Let $n>1.$ Is it consistent with $GCH$ to
have $\KH(\aleph_n, \aleph_0)$ but not $\KH(\aleph_n, \aleph_1)$?
\end{question}
\begin{question} Let $n >1$. Is it consistent with $GCH$ to
have $\KH(\aleph_n, \aleph_i)$ for all $i<n,$ but not $\KH(\aleph_n,
<\aleph_n)?$
\end{question}

{Kurt G\"{o}del Research Center, University of Vienna,

E-mail address: sdf@logic.univie.ac.at}

{Department of Mathematics, Shahid Bahonar University of Kerman,
Kerman-Iran and School of Mathematics, Institute for Research in
Fundamental Sciences (IPM), Tehran-Iran.

E-mail address: golshani.m@gmail.com}


\begin{thebibliography}{xx}


\bibitem{ } Devlin, K.J., Constructibility, Perspectives in mathematical
logic, 1984.

\bibitem{ } Jech, T., Set theory, Academic Press, 1978.


\bibitem{ } Jensen, R., Some combinatorial properties of $L$ and
$V$, http://www.mathematik.hu-berlin.de/~raesch/org/jensen.html.

\bibitem{ } Jensen, R., Remarks on the Two Cardinal Problem, http://www.mathematik.hu-berlin.de/~raesch/org/jensen.html.

\bibitem{ } Mitchell, W., Aronszajn trees and the independence of
the transfer property, Ann. Math. Logic 5 (1972), 21-46.

\bibitem{ } Raesch, T., On the Failure of the GAP-1 Transfer
Property, PhD thesis, http://www.math.uni-bonn.de/people/raesch/publicationen.html

\bibitem{ } Silver, J, The independence of Kurepa's conjecture and
two cardinal conjectures in model theory, In ``Axiomatic Set
Theory'' Proc. Symp. Pure Math. 13,1 (D. Scott, ed.)pp. 383-390.
Amer. Math. Soc., Providence Rhode Island, 1971
\end{thebibliography}
\end{document}